\documentclass[12pt]{amsart}

\usepackage{amsfonts}
\usepackage[usenames,dvipsnames]{xcolor}
\usepackage{amscd}
\usepackage{color}
\usepackage{epsfig}
\usepackage{graphicx}
\usepackage{amsmath,amssymb}
\usepackage{xypic}
\usepackage{enumerate}
\usepackage{amsthm}
\usepackage{url}

\newtheorem{thm}{Theorem}[section]
\newtheorem{lem}[thm]{Lemma}

\newtheorem{prop}[thm]{Proposition}
\newtheorem{cor}[thm]{Corollary}

\newtheorem{rmk}[thm]{Remark}

\newtheorem{ques}[thm]{Question}
\newtheorem{prob}[thm]{Problem}

\newtheorem{ex}[thm]{Example}

\def\O{{\mathcal O}}
\def\L{{\mathcal L}}
\def\M{{\mathcal M}}
\def\P{{\mathbb P}}

\def\I{{\mathcal I}}

\def\Z{{\mathbb Z}}

\def\C{{\mathbb C}}
\def\Pthree{{\mathbb P}^3}

\def\Cl{\mathop{\rm Cl}}

\def\Spec{\mathop{\rm Spec}}
\def\codim{\mathop {\rm codim}}

\def\mod{\mathop{\rm mod}}
\def\m{\mathop{\rm m}}
\def\fm{\mathfrak m}
\def\depth{\mathop{\rm depth}}

\def\mult{\mathop{\rm mult}}

\def\ra{\rightarrow}

\newcommand{\ubm}[2]{\underbrace{#1}_{#2}}
\newcommand{\cyc}[1]{\langle {#1} \rangle}
\def\rdA{\mathbf A}
\def\rdD{\mathbf D}
\def\rdE{\mathbf E}

\title{Srinivas' problem for rational double points}

\author{John Brevik}

\address{California State University at Long Beach, 
Department of Mathematics and Statistics, Long Beach, CA 90840}

\email{jbrevik@csulb.edu}

\author{Scott Nollet}

\address{Texas Christian University, Department of Mathematics, 
Fort Worth, TX 76129}

\email{s.nollet@tcu.edu}

\subjclass[2000]{Primary: 14B07, 14H10, 14H50}

\begin{document}
\bibliographystyle{plain}

\begin{abstract} 
For the completion $B$ of a local geometric normal domain, 
V. Srinivas asked which subgroups of $\Cl B$ arise as the image of 
the map $\Cl A \to \Cl B$ on class groups as $A$ varies among 
normal geometric domains with $B \cong \hat A$. 
For two dimensional rational double point singularities we show that 
all subgroups arise in this way. 
We also show that in any dimension, every normal hypersurface 
singularity has completion isomorphic to that of a geometric UFD. 
Our methods are global, applying Noether-Lefschetz theory to linear 
systems with non-reduced base loci.
\end{abstract}

\maketitle

\section{Introduction}

V. Srinivas posed several interesting problems about class groups of 
noetherian local normal domains in his survey paper on geometric methods in 
commutative algebra \cite[$\S 3$]{srinivas}. 
Recall that if $A$ is such a ring with completion $\hat A$, 
then there is a well-defined injective map on divisor class groups 
$j: \Cl A \to \Cl \hat A$~\cite[\S1, Proposition 1]{samuel} arising from valuation theory. 
For geometric local rings, {\em i.e.}, localizations of 
$\mathbb C$-algebras of finite type, Srinivas asks about the 
possible images of the map $j$
\cite[Questions 3.1 and 3.7]{srinivas}: 

\begin{prob}\label{one}{\em 
Let $B$ be the completion of a local geometric normal domain. 
\begin{enumerate}
\item[(a)] What are the possible images of $\Cl A \hookrightarrow \Cl B$ as $A$ ranges 
over all geometric local normal domains with $\widehat A \cong B$?
\item[(b)] Is there a geometric normal local domain $A$ with $\hat A \cong B$ and 
$\Cl A = \langle \omega_B \rangle \subset \Cl B$?
\end{enumerate}
\em}\end{prob}

While we are mainly interested in (a), let us  
review the progress on Problem \ref{one}(b). Since the dualizing 
module $\omega_B$ is necessarily in the image of 
$\Cl A \hookrightarrow \Cl B$ whenever $A$ is a quotient of a 
regular local ring \cite{murthy}, part (b) asks whether the 
{\it smallest} {\em a priori} image is possible. 
Moreover, if $B$ is Gorenstein, then $\omega_B$ is trivial in $\Cl B$ 
and part (b) asks whether $\Cl A = 0$ is possible; in other words, whether 
is $B$ the completion of a unique factorization domain (UFD). 
For arbitrary rings, Heitmann proved that $B$ is the completion of a UFD if and 
only if $B$ is a field, a discrete valuation ring, or $\dim B \geq 2$, 
$\depth B \geq 2$ and every integer is a unit in $A$ \cite{heit1}, 
but for $\dim A \geq 2$ his constructions produce rings that are far 
from geometric.  
 
For Problem \ref{one} (b) as stated, perhaps the most famous 
result is Grothendieck's solution to Samuel's conjecture 
\cite[XI, Corollaire 3.14]{SGA2}, which says that if $B$ is a complete intersection which is factorial in codimension $\leq 3$ 
%\footnote{{\em i.e.}, $B_p$ is a UFD for all primes $p$ of height $\leq 3$
%\blu{Do you think this footnote is necessary? We could probably safely remove it.}}, 
then $\Cl B = 0$ so that $B$ is already a UFD. In particular, 
any complete intersection ring $B$ of dimension $\geq 4$ with an isolated singularity is a UFD. 
Parameswaran and Srinivas showed that such rings $B$ of dimension 
$d = 2,3$ are completions of UFDs \cite{parasrini}, 
extending the earlier result of Srinivas for rational double points \cite{srinivas1}. 
Hartshorne and Ogus proved that any ring $B$ with an isolated singularity and depth $\geq 3$ 
having codimension $\geq 3$ in a regular local ring is a UFD \cite{HO}. 
Parameswaran and Van Straten answered Problem \ref{one} (b) positively for 
arbitrary normal surface singularities \cite{paravanstraten}, the only result 
that discusses non-trivial subgroups. 
We offer the following:

\begin{thm}\label{ufd} 
Let $B$ be a completed normal hypersurface singularity.
Then there exists a hypersurface $X \subset \P^n_\C$ and a 
point $p \in X$ such that $A=\O_{X,p}$ is a UFD and $\widehat A \cong B$. 
\end{thm} 

Unlike the results mentioned above for isolated singularities, our proof 
is very short (about a page) and uniformly handles all dimensions, 
though it only addresses rings $B$ of the form 
$\mathbb C[[x_1,\dots,x_n]]/(f)$ where $f$ defines a variety $V$ 
normal at the origin. We expect to extend our method 
to all normal local complete intersection singularities. 

On the other hand, we have seen no work addressing the more difficult 
Question \ref{one} (a) since it appeared ten years ago. Noting an example 
of a ring $B$ for which $\Cl B \cong \mathbb C$ but the images $\Cl A$ 
are finitely generated \cite[Example 3.9]{srinivas}, Srinivas 
seems pessimistic about the chances, saying that Problem \ref{one} (b) is 
``probably the only reasonable general question in the direction of 
Problem \ref{one} (a)", but his example at least suggests the following: 

\begin{ques}\label{three}\label{finite}{\em
For $B$ as above, is every finitely generated subgroup 
$H \subset \Cl B$ containing $\langle \omega_B \rangle$ the image of 
$\Cl A$ for a local geometric normal domain $A$ with $\widehat A \cong B$? 
\em}\end{ques}

Our first result is a positive answer for the best-understood surface singularities: 

\begin{thm}\label{ratdouble}
Let $B$ be the completion of the local ring of a rational double point on a surface. 
Then for any subgroup $H \subset \Cl B$, there is a local geometric normal ring $A$ 
with $B \cong \hat A$ and $H = \Cl A \subset \Cl B$. 
\end{thm}

There is a well known classification of the surface rational double points, 
namely $\rdA_n, \rdD_n, \rdE_6, \rdE_7, \rdE_8$ \cite{lipman} and in each case 
we produce an actual algebraic surface 
$S \subset \Pthree_\C$ and a rational 
double point $p \in S$ such that $\Cl \O_{S,p} \cong H$. 
This result comes as a surprise in view of Mohan Kumar's result \cite{MK}. 
He proved that for almost all $\rdA_n$ and $\rdE_n$ type singularities 
on a {\em rational} surface over $\C$, the analytic isomorphism 
class {\it determines} the algebraic isomorphism class; the exceptions are 
$\rdA_7, \rdA_8$ and $\rdE_8$, for which there are two possibilities 
each. In particular, the possibility for $\Cl (A) \hookrightarrow \Cl (\widehat A)$ 
is unique, except for these cases (and also the $\rdE_8$ case, since the complete local ring is a UFD). 

While the results above are local algebraic, our method of proof uses global algebraic geometry. 
Our idea in each case is to exhibit a base locus 
$Y \subset \mathbb P^n$ and a point $p \in Y$ constructed so that the 
general hypersurface $X$ containing $Y$ has the desired singularity type 
at $p$, meaning that $B \cong \widehat \O_{X,p}$. 
Taking $A = \O_{X,p}$, an honest local ring from the variety $X$, we can read 
off the generators of $\Cl A$ by the following consequence of our extension 
of the Noether-Lefschetz theorem to linear systems with base locus \cite{BN}. 

\begin{thm}\label{gens} Let $Y \subset \P_\C^n$ be a closed subscheme which is properly 
contained in a normal hypersurface and suppose $p \in Y$. Then the very general hypersurface 
$X$ of degree $d \gg 0$ containing $Y$ is normal and $\Cl \O_{X,p}$ is generated by 
supports of the codimension $2$ irreducible components of $Y$. 
In particular, $\codim Y > 2 \Rightarrow \Cl \O_{X,p} = 0$. 
\end{thm} 

After applying Theorem \ref{gens}, we use power series techniques to compute 
the images of the local class groups to obtain the results. 
For example, the proof of Theorem \ref{ufd} combines Theorem \ref{gens} with 
a power series lemma due to Ruiz \cite{ruiz}. 
The constructions in Theorem \ref{ratdouble} are more complicated and rather interesting. 
For $\rdA_n$ singularities we used a Cohen-Macaulay multiplicity structure 
on a smooth curve for $Y$, but for $\rdD_n$ singularities we found it necessary 
to use a line with an embedded point at $p$. In the case of a $\rdD_n$ 
singularity given locally by $x^2 + y^2 z + z^{n-1} = 0$, the completed local 
ring has class group $\Z / 4 \Z$ when $n$ is odd and 
$\Z / 2 \Z \oplus \Z / 2 \Z$ when $n$ is even. 
In the odd case there is only one subgroup of order $2$, 
but in the even case there are {\em three} of them, 
two of which are indistinguishable up to automorphism of the complete local ring 
since their generators correspond to conjugate points in the associated 
Dynkin diagram \cite{J2}. In this case we prove the stronger statement that 
each of the two non-equivalent subgroups arises as the class group of a 
local ring on a surface. Finally, in Proposition~\ref{smoothrep} we prove that any divisor class of a rational double point surface singularity has a representative that is smooth at that point.

Regarding organization, we prove Theorem \ref{gens} and Theorem \ref{ufd} 
in sections 2 and 3, respectively. Section 4 is devoted to Theorem \ref{ratdouble}. 
We work throughout over the field $k = \C$ of complex numbers 
except as noted. While the base locus $Y$ should be projective to apply 
Theorem \ref{gens}, we often give a local ideal for $Y \subset \mathbb A^n$ 
and apply the theorem to $\overline Y \subset \P^n$. 

\section{Geometric generators of Local class groups}

In this section we prove Theorem~\ref{gens} from the introduction, 
which identifies the generators of the local class group of a general member 
of a linear system of hypersurfaces at a point lying in the base locus. 
The proof is quite short, being a straightforward consequence of our 
Noether-Lefschetz theorem with base locus \cite{BN} and we present it 
without further ado.  

\begin{proof} The condition that $Y$ is properly contained in a normal hypersurface 
implies that $Y$ has codimension $\geq 2$ and its points of embedding dimension $\geq n$ 
have codimension $\geq 3$, which is to say that $Y$ is {\it superficial} in the language 
used in \cite{BN} (the converse is also easy to see). Thus we may apply 
\cite[Theorems 1.1 and 1.7]{BN} to see that $X$ is normal and $\Cl X$ is (freely) 
generated by $\O_X (1)$ and the supports of the codimension two components of $Y$. 
On the other hand, the natural restriction map $\Cl X \to \Cl \O_{X,p}$ is surjective 
(because height one primes in $\O_{X,p}$ lift to global Weil divisors on $X$), 
so $\Cl \O_{X,p}$ is generated by these same classes. 
Since $\O_X (1)$ and the supports of the codimension two components of $Y$ not passing 
through $p$ have trivial restriction in $\Cl \O_{X,p}$, it follows that $\Cl \O_{X,p}$ is 
generated by the remaining supports of codimension two components of $Y$ which pass through $p$.
\end{proof} 

One can use Theorem~\ref{gens} together with the natural injection $\Cl \O_{X,p} \to \widehat{\Cl \O_{X,p}}$ to calculate class groups of local rings by carrying out calculations inside power series rings. We will use this idea to prove our main results; for further applications, see~\cite{picgps}. The following restates Theorem~\ref{gens} in algebraic terms. 

\begin{cor}
Let $I \subset R = \C[x_1,\dots,x_n]$ be an ideal of height $\geq 2$. 
In the primary decomposition $I = \bigcap q_i$ with $q_i$ a $p_i$-primary ideal, 
assume that $q_i \not \subset p_i^2$ for each height two prime $p_i$. 
Then for the very general $f \in I$ and $A = \C[x_1, \dots, x_n]/(f)$, 
the image of $\Cl (A) \hookrightarrow \Cl (\widehat A)$ 
is generated by the height two primes associated to $I$. 
\end{cor}

\section{Completions of unique factorization domains} 
Given a complete local ring $A$, what is the nicest ring $R$ having completion $A$? 
Heitmann shows that $A$ is the completion of a UFD if and only if $A$ is a field, $A$ is a DVR, or $A$ 
has depth $\geq 2$ and no integer is a zero-divisor of $A$ \cite{heit1}; however, his constructions 
need not lead to excellent rings. Loepp shows with minimal hypothesis that $A$ is the 
completion of an {\em excellent} local ring \cite{loepp}, though her construction need not produce a UFD. 
Using geometric methods, Parameswaran and Srinivas show that the completion of the local ring at an isolated 
local complete intersection singularity is the completion of a UFD which is the local ring for a variety \cite{parasrini}. 
We prove the same for normal hypersurface singularities over $\C$ that need not be isolated. We first need a lemma, which can be found in Ruiz' book on power series \cite[V, Lemma 2.2]{ruiz}

\begin{lem}\label{Ruiz} Let $\fm \subset \C [[x_1, \dots, x_n]]$ denote the maximal ideal, fix 
$f \in \fm^2$, and define $J_f=(f_{x_1}, \dots, f_{x_n})$ to be the ideal generated by the partial 
derivatives of $f$. Then for any $g \in \C[[x_1, \dots, x_n]]$ such that
$f-g \in \fm \cdot J_f^2, \C[[x_1, \dots x_n]]/(f) \cong \C[[x_1, \dots, x_n]]/(g)$.
\end{lem}

\begin{rmk}{\em 
Ruiz actually proves Lemma 3.1 for the ring $\mathbb C\{x_1, \dots x_n\}$ of convergent power series and notes that the same proof goes through in the formal case above.
\em}
\end{rmk}

We proceed to prove Theorem~\ref{ufd}.

\begin{proof}
Let $f \in \mathbb C [x_1,x_2, \dots x_n]$ be the equation of a hypersurface $V$ which is 
singular and normal at the origin $p$, corresponding to the maximal ideal $\fm = (x_1, \dots, x_n)$, and let 
$B = \mathbb C [[x_1,x_2, \dots x_n]]/(f)$ be the completion of $\O_{V,p}$. 
The singular locus $D$ of $V$ is given by the ideal $(f) + J_f$, 
where $J_f = (f_{x_1}, \dots, f_{x_n})$. Using primary decomposition in the ring $\mathbb C [x_1, \dots, x_n]$, 
we may write 
\[
J_f =\bigcap_{p_i \subset \fm} q_i \, \cap \, \bigcap_{p_i \not \subset \fm} q_i
\]
where $q_i$ is $p_i$-primary and we have sorted into components that meet 
the origin and those that do not. Denote by $K$ the intersection on the left and $J$ the intersection on the right; localizing at $\fm$ we find that 
$(J_f)_\fm = K_\fm$ because $J_\fm = (1)$.

Now if $K = (k_1, \dots, k_r)$ gives a polynomial generating set for $K$, 
the closed subscheme $Y$ defined by the ideal $I_Y=(f, k_1^3, \dots k_r^3)$ 
is supported on the components of the singular locus of $F = \{f=0\}$ which 
contain the origin, hence has codimension $\geq 3$ by normality of $F$ at 
the origin. The very general hypersurface $X$ containing $Y$ satisfies 
$\Cl \O_{X,p} = 0$ by Theorem \ref{gens}, so $\O_{X,p}$ is a UFD~\cite[Prop. 6.2]{AG}. 
Moreover, $X$ has local equation
\[
g = f + a_1 k_1^3 + \dots + a_r k_r^3
\]
for units $a_i$, and clearly $f-g \in K^3$. Since $K_\fm = (J_f)_\fm$, 
their completions are equal in $\mathbb C [[x_1, \dots, x_n]]$. 
Therefore $f-g \in J_f^3 \subset \fm J_f^2$, and it follows from Lemma~\ref{Ruiz} that
$\widehat \O_{X,p} = \C[[x_1, \dots, x_n]]/(g) \cong \C[[x_1, \dots, x_n]]/(f) = B$.
\end{proof}

\begin{ex}{\em For an {\em isolated} singularity, one can give the following proof based on the Mather-Yau theorem~\cite{MY}. 
The ideal $I=(f, f_{x_1}, \dots, f_{x_n})$ generated by $f$ and its partial derivatives defines a $0$-dimensional 
scheme $Y$ supported at the origin, hence $(x_1, \dots, x_n)^N \subset I$ for some $N > 0$. The scheme $Z$ defined by 
$(f, x_1^{N+2}, \dots, x_n^{N+2})$ is also supported at the origin $p$, so by Theorem \ref{gens} the very general 
surface $S$ containing $Z$ satisfies $\Cl \O_{S,p} = 0$. The local equation of $S$ has the form 
$g = f+\sum a_i x_i^{N+2}$ for units $a_i$ in the local ring $\O_{\P^n,p} \cong \C[x_1, \dots, x_n]_{(x_1, \dots, x_n)}$. 
Observe that $J = (g,g_{x_1}, \dots, g_{x_n}) \subset I$ because $g-f$ and its partials lie in 
$(x_1, \dots x_n)^{N+1} \subset I$. These ideals are equal because the induced map $J \to I/(x_1, \dots x_n)I$ 
is obviously surjective, therefore so is the map $J \to I$ by Nakayama's lemma. It follows that $I = J$ in the 
ring $\C\{x_1, \dots, x_n\}$ of germs of holomorphic functions as well, so $f$ and $g$ define 
(complex-) analytically isomorphic singularities; this isomorphism lifts to a formal-analytic isomorphism.
We have thus produced a UFD, namely $\O_{S,p}$, whose completion is isomorphic to $A$.
\em}\end{ex}

\section{Local Class Groups of Rational Double Points}

As noted in the introduction, Mohan Kumar showed that for $\rdA_n$ and $\rdE_n$ double points 
on a {\em rational} surface, the analytic isomorphism class determines the algebraic isomorphism 
class of the local ring, with the three exceptions $\rdA_7, \rdA_8, \rdE_8$ for which there are 
two possibilities each. Regarding the question of Srinivas mentioned in the introduction, this means that there 
is one (or sometimes two) possibilities for the inclusion $\Cl A \hookrightarrow \Cl \widehat A$ for the 
corresponding local rings. Our main goal in this section is to prove Theorem \ref{ratdouble} 
from the introduction, which says that without the rationality hypothesis, {\em all} subgroups arise in this way. 
Note that the case of the trivial class group is Theorem~\ref{ufd}, so we only need prove the result for nontrivial class groups.

We will prove this for each singularity type $\rdA_n, \rdD_n, \rdE_n$ separately, noting their standard equations 
as we will use, the class group $\Cl \widehat A$ of the completion of a local ring $A$ having the given type, 
and the Dynkin diagram of exceptional $(-2)$-curves for the minimal resolution (these have been known for a 
long time; {\em cf.} \cite{artin,lipman}). The class of any curve is determined by its intersection numbers with 
the exceptional curves $E_j$; denote by $e_j$ the class-group element corresponding to a curve 
meeting $e_j$ once and no other $e_i$.

The approach in the interesting cases is to find a base locus that forces both the desired singularity type and a curve, guaranteed by Theorem~\ref{gens} to generate the class group, whose strict transform has the desired intersection properties with the exceptional locus.

\subsection{$\rdA_n$ singularities} The $\rdA_n$ singularity for $n\ge 1$ is analytically isomorphic to that given by the equation 
$xy - z^{n+1}$ at the origin; its class group is $\Z/(n+1)\Z$, and the Dynkin 
diagram is 
$$\entrymodifiers={!! <0pt, .6ex>+}\xymatrix{
\underset{E_1}\circ\ar@{-}[r] & \underset{E_2}\circ\ar@{-}[r] & \dots & \underset{E_n}{\circ}\ar@{-}[l] }
$$
Here $e_1$ is one generator, and $e_j = je_1$ in the local class group.

The following, which immediately implies Theorem~\ref{ratdouble} for $\rdA_n$ singularities, is Prop. 4.1 from ~\cite{picgps}; there we used it to calculate the class groups of general surfaces containing base loci consisting of certain multiplicity structures on smooth curves.

\begin{prop}\label{step1}
Let $Z \subset \Pthree_\C$ be the subscheme with ideal 
$\I_{Z} = (x^2,xy,x z^q - y^{m-1}, y^m)$ for $m \geq 3$. 
Then the very general surface $S$ containing $Z$ has an $\rdA_{(m-1)q-1}$ singularity at 
$p=(0,0,0,1)$ and $\Cl(\O_{S,p}) \cong \Z / (m-1) \Z$ is generated by $C$.    
\end{prop}

\subsection{$\rdD_n$ singularities}
The $\rdD_n$ singularity for $n\ge 4$ is analytically isomorphic to that given by $x^2+y^2 z + z^{n-1}$ at the origin. The class group of the completion is $\Z/4\Z$ if $n$ is even and $\Z/2\Z \oplus \Z/2\Z$ if $n$ odd, and the Dynkin diagram is
$$\entrymodifiers={!! <0pt, .6ex>+}\xymatrix{ & & &  \overset{E_{n-1}}{\circ}\ar@{-}[d] & \\
\underset{E_1}{\circ} \ar@{-}[r] & \underset{E_2}\circ\ar@{-}[r] & \dots & \underset{E_{n-2}}{\circ}\ar@{-}[l] \ar@{-}[r]  & \underset{E_n}{\circ} 
}
$$
Here, $e_j$ for $j<n-1$ is $je_1$ as in the $\rdA_n$ case. If $n$ is even, $E_{n-1}$ and $E_n$ are the two distinct generators of $\Cl \hat A$, and $e_1$ corresponds to the element $2$; if $n$ is odd, then $e_1, e_{n-1}$, and $e_n$ are the three distinct nonzero elements.

We will need the following lemmas.

\begin{lem}\label{comproots} \cite[Lemma 2.2]{picgps} Let $(R,\fm)$ be a complete local domain, and let $n$ be a positive 
integer that is a unit in $R$. If $a_0 \in R$ is a unit and $u \equiv a_0^n$ mod $\fm^k$ 
for some fixed $k > 0$, then there exists $a \in R$ such that $a^n = u$ and $a \equiv a_0$ mod $\fm^k$.
\end{lem}

\begin{lem}\label{dnen}
Let $R = k[[y,z]]$ with maximal ideal $\fm \subset R$. 
For integers $a,s,t$ with $s > a > 1, t > a+1$ and $b \in R$ a unit, 
there is a change of coordinates $Y,Z$ such that 
\[
f=y^{a} z + z^{s} - b y^{t} = Y^{a} Z + Z^{s}.
\]
Furthermore $X,Y$ may be chosen so that $y \equiv Y$ mod $\fm^2$ and $z \equiv Z$ mod $\fm^2$.
\end{lem}

\begin{proof}
We produce coordinate changes $y_i, z_i$ such that 
$y_{i+1}\equiv y_i \mod \fm^{i+1}, z_{i+1}\equiv z_i \mod \fm^i$ 
and $f = y_i^a z_{i} + z_i^s -b_i y_{i}^{k_i}$ with 
$k_i\ge i+a+1$ and $b_{i}$ a unit. By hypothesis $y_1=y, z_1=z$ give the base step $i=1$.

For the induction step, let $z_{i+1} = z_i - b_i y^{k_i-a} \equiv z_i\mod \m^{i+1}$ so that
$$f=y_i^a z_{i+1} + (z_{i+1}^s + sb_iz_{i+1}^{s-1}y_i^{k_i-a} + 
\dots + sb_i^{s-1}z_{i+1}y_i^{(s-1)(k_i-a)} + y_i^{s(k_i-a)})$$
$$=y_i^a z_{i+1}\ubm{[1+sb_iz_{i+1}^{s-2}y_i^{k_i-2a} + \dots + 
sb_i^{s-1}y_i^{(s-1)(k_i-2a)}]}{v_{i}} + 
z_{i+1}^s + b_i^s y_i^{s(k_i-a)}$$
where $v_i$ is a unit with lowest-degree term after the leading $1$ is of degree 
$s-2+k_i-2a \ge i$. By Lemma~\ref{comproots}, $v_i$ has an $a^{\rm th}$ root 
$w_i$ that is congruent to $1$ mod $\fm^i$. 
Then $y_{i+1} = w_i y_i \equiv y_i \mod \fm^{i+1}$, so that 
$f = y_{i+1}^a z_{i+1} + z_{i+1}^s - b_{i+1} y_{i+1}^{k_{i+1}}$, 
where $b_{i+1} = -b_i^s w_i^{-s(k_i-a)}$ is a unit and 
$k_{i+1} = s(k_i-a) \ge s(i+1) \ge (a+1)(i+1) \ge a+i+2$, completing the induction.
\end{proof}

\begin{prop} Theorem~\ref{ratdouble} holds for $\rdD_n$ singularities.
\end{prop}

\begin{proof}
The class group of a complete $\rdD_n$ singularity is either $\Z / 4 \Z$ 
or $\Z / 2 \Z \oplus \Z / 2 \Z$; we first produce the subgroups of order two as class groups of local rings of surfaces.  The scheme $Z$ defined by $I_Z = (x^2,y^2z, z^{n-1}, xy^n)$ consists of the line $L: x=z=0$ 
and an embedded point at the origin $p$, hence the very general surface $S$ containing $Z$ 
has local Picard group $\Cl(\O_{S,p})$ generated by $L$ by Theorem \ref{gens} 
and $L \neq 0$ in this group because $L$ is not Cartier at $p$, for $L$ is smooth at $p$ while $S$ is not. 
The local equation of $S$ has the form $ax^2 + by^2z + z^{n-1} + cxy^n$ with units 
$a, b, c \in \O_{\Pthree,p}$. After we use Lemma 4.2 to take square roots of $a,b$ in $\widehat \O_{\Pthree,p}$ the equation becomes 
\[ 
x^2 + y^2z + z^{n-1} + cxy^n=(\ubm{x+\frac{c}{2}y^n}{x_1})^2 + y^2z + z^{n-1} - \frac{c^2}{4} y^{2n} 
\] 
and applying the coordinate change of Lemma \ref{dnen} exhibits the $\rdD_n$ singularity.

To determine the class of $L$, we blow up $S$ at $p$. The local equation of $\widetilde S$ on the 
patch $Y=1$ is $aX^2 + byZ + y^{n-3}Z^{n-1} + cXy^{n-1}$ and the strict transform $\widetilde L$ 
has ideal $(X,Z)$, hence $\widetilde L$ meets the exceptional curve at the new origin of this patch, 
which is the $\rdA_1$ singularity whose blow-up will produce the exceptional divisor $E_1$; 
Resolving this $\rdA_1$ shows that $\widetilde L$ meets $E_1$ but not $E_2$, so $L$ gives the class 
$u_1$ defined above. In particular, $2L = 0$ and $\Cl(\O_{S,p}) \cong \Z / 2 \Z$. 

When $n$ is odd, $\cyc{u_1}$ is the only subgroup of order $2$ in $\Cl(\O_{S,p})$ and we are 
finished; when $n$ is even -- at least when $n \geq 6$ -- there are three such subgroups 
$\cyc{u_1}, \cyc{u_{n-1}}$ and $\cyc{u_n}$, the last two being distinguishable from the first, 
but not from one another, as they correspond to the two exceptional curves in the final blow-up. 
We therefore construct a $\rdD_n$ singularity, $n$ even, such that its local Picard group is 
generated by $u_{n-1}$ or $u_n$. 

Write $n=2r$ and define $Z$ by the ideal $I_Z = (x^2, y^2z-z^{2r-1}, y^5-z^{5r-5})$. 
The last two generators show that $y \neq 0 \iff z \neq 0$ along $Z$, when $I_Z$ is 
locally equal to $(x^2, y - z^{r-1})$, thus $Z$ consists of a double structure on the 
smooth curve $C$ with ideal $(x,y-z^{r-1})$ and an embedded point at the origin $p$. 
Therefore if $S$ is a very general surface containing $Z$, then as before $C$ generates 
$\Cl(\O_{S,p})$ and is nonzero. The local equation for $S$ has the form 
$ax^2 + y^2z-z^{2r-1} + by^5-bz^{5r-5}$ for units $a,b \in \O_{\Pthree,p}$. 
Passing to the completion and adjusting the variables by appropriate roots of units, the equation 
of $S$ becomes $x^2+y^2z+z^{2r-1}+y^5$. Changing variables via Lemma \ref{dnen}, the equation becomes 
$x^2 + Y^2 Z + Z^{2r-1}$ and we see the $\rdD_{2r}$-singularity. 

To determine the class of $C$ in $\Cl(\O_{S,p})$, we return to the original form of the 
equation for $S$ and blow up $p$. Following our usual conventions, look on the patch $Z=1$, 
where $S$ has equation $X^2 + aY^2z - az^{2r-3} + bY^5z^3 - bZ^{5r-7}$. As long as $2r-3 \ge 3$, 
that is, $r\ge 3$, we can take advantage of our knowledge of the original singularity, together 
with the fact that this expression is a square mod the third power of the maximal ideal at the origin, 
to conclude that it is this new surface has a $\rdD_{2r-2}$ singularity there; the strict transform 
$\tilde{C}$ of $C$ has ideal $(X,Y-z^{r-3})$ and so passes through the origin transversely to the 
exceptional curve $E_1$ given by $(X,z)$. Now, on the full resolution of singularities $\tilde{C}$ 
maps to a smooth curve and therefore meets exactly one admissible $E_i$, that is, either $E_1$, $E_{n-1}$, 
or $E_n$. By the transversality noted above, however, it does not meet $E_1$ in the next blowup. 
We conclude that it meets one of the other curves and so we have produced the desired subgroup.

Finally, we need an example for which $\Cl(\O_{S,p}) = \Cl(\widehat \O_{S,p})$. To this end consider 
the surface $S$ defined by $x^2 + y^2 z - z^{n-1}$, which has a $\rdD_n$ singularity at the origin 
(change coordinates $\displaystyle{z \mapsto z^\prime = e^\frac{\pi i}{n-1}z, y \mapsto y^\prime = e^{-\frac{\pi i}{2(n-1)}y}}$). As above, the curve with ideal $(x,z)$ corresponds to $u_1$, so it suffices 
in all cases to find a curve on this surface that gives the element $u_n$ (or $u_{n-1}$). 

For even values of $n$, the curve $\displaystyle{(x,y-z^\frac{n-2}{2})}$ gives the other generator by an argument 
analogous to, but much easier than, the one for $n$ even in the previous case considered.

For odd values of $n$, write $n=2r+1$, use the form $x^2+y^2z-z^{2r}$, and let $C$ be the curve having 
ideal $(x-z^r, y)$. Looking on the patch $Z=1$ of the blow-up gives the surface $X^2 + Y^2z -2z^{2r-2}$ 
and the curve $(X-z^{r-1}, Y)$, so by induction it suffices to prove that, for $n=5$, that is, $r=2$, 
the curve $C$ having ideal $(x-z^2, y)$ gives one of the classes $u_4, u_5$ in $\Cl(\O_{S,p})$, where $S$ 
has equation $x^2+y^2z-2z^4$ and $p$ is the origin. On the patch $Z=1$ of the blow-up $\tilde{S}$, 
the equation is $X^2+Y^2z-Z^2$, which has the $\rdA_3$ at the origin, and $\tilde{C}$ has ideal $(X-z, Y)$. 
A further blow-up, again  on the patch $Z=1$, gives equation $X^2+Y^2z-1$ for $\tilde{S}$ and ideal $(X-1, Y)$, 
for $\tilde{C}$. $C$ meets the exceptional locus only at the smooth point with coordinates $(1,0,0)$ relative 
to this patch, so in the full resolution of singularities, $\tilde{C}$ meets one of the two exceptional 
curves arising from the blow-up of the $\rdA_3$, which correspond to $E_4$ and $E_5$ in the Dynkin diagram.
\end{proof}

\subsection{$\rdE_6, \rdE_7, \rdE_8$ singularities} The standard equations for the analytic isomorphism types and class groups for these singularities are as follows:
\begin{center}
\begin{tabular}{|l|l|l|}
\hline
Type & Equation & $\Cl \hat A$ \\
\hline \hline
$\rdE_6$ & $x^2+y^3+z^4$ & $\Z/3\Z$ \\
\hline
$\rdE_7$ & $x^2+y^3+yz^3$ & $\Z/2\Z$ \\
\hline
$\rdE_8$ & $x^2+y^3+z^5$ & $0$ \\
\hline
\end{tabular}
\end{center}

and the $\rdE_n$ Dynkin diagram looks like
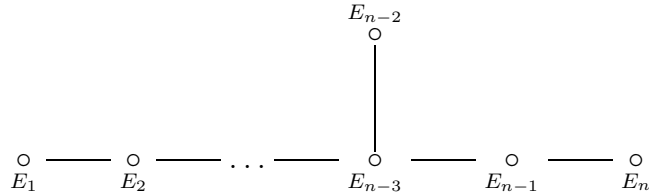
\[ \entrymodifiers={!! <0pt, .6ex>+}
\xymatrix{ & & &  \overset{E_{n-2}}{\circ}\ar@{-}[d] & \\
\underset{E_1}{\circ} \ar@{-}[r] & \underset{E_2}\circ\ar@{-}[r] &
\dots & \underset{E_{n-3}}{\circ}\ar@{-}[l] \ar@{-}[r]  & 
\underset{E_{n-1}}{\circ} \ar@{-}[r]  & \underset{E_n}{\circ}} 
\]
For our purposes, as it happens, the way that the exceptional curves correspond to class group elements is not important.

\begin{prop} Theorem~\ref{ratdouble} holds for $\rdE_6, \rdE_7$, and $\rdE_8$ singularities.
\end{prop}

\begin{proof}
We may assume $T \neq \rdE_8$ because in that case the class group is automatically $0$. If $T = \rdE_6$, 
we take the affine quartic surface $S$ with equation $x^2+y^3+z^4=0$, for which the $\rdE_6$ 
singularity at the origin $p$ is apparent. The smooth curve $C$ with ideal $(x-i z^2, y)$ lies 
on $S$ and passes through $p$. If $C$ restricts to $0$ in $\Cl \O_{S,p}$, then $C$ is Cartier on 
$S$ at $p$, impossible because $C$ is smooth at $p$ while $S$ is not: thus $C$ defines a non-zero 
element in $\Cl \O_{S,p}$, which must generate all of $\Cl\widehat \O_{S,p} \cong \Z / 3 \Z$, the only 
non-trivial subgroup. Similarly, for $T = \rdE_7$ we take the affine quartic $S$ defined by the 
equation $x^2+y^3+yz^3$ which contains the smooth $z$-axis $C$ passing through the origin $p$. 
\end{proof}

In~\cite[Remark 5.2.1]{Zeuthen}, Hartshorne points out that in the local ring of an $\rdA_n$ singularity with the specific equation $xy-z^{n+1}$, every element of the class group of the local ring is represented by a curve that is smooth at that point. We now show that there is a locally smooth representative for any divisor class of any rational double point surface singularity.

\begin{prop}\label{smoothrep} For any surface $S$, if $p\in S$ such that $\O_{S,p}$ defines a rational double point, then {\em each} class of  $\Cl \O_{S,p}$ is given by a curve that is smooth at $p$. 
\end{prop}

\begin{proof}
Let $\pi:X\ra S$ be the minimal resolution of the singularity at $p$, with exceptional curve $E=\bigcup E_i$, and let $C$ be a curve on $S$; we wish to find a locally smooth curve $C'$ on $S$ giving the same class as $C$ in $\Cl \O_{S,p}$. Now let $\xi_0$ be the {\em fundamental cycle} on $X$ ({\em Cf.} \cite{artin}, pp. 131-2), which is an effective, exceptionally supported divisor with the property that $\xi_0$ has nonpositive intersection with each exceptional curve and is the least nonzero divisor with this property. Using the $n$-tuple $(a_1, \dots, a_n)$ as shorthand for $a_1E_1+\dots + a_nE_n$, we note the fundamental cycles for the various rational double points:
$${\mathbf A}_n: (1,1,\dots, 1). \,\, {\mathbf D}_n: (1, 2, 2, 2, . . . , 2, 1, 1).\,\, {\mathbf E}_6: (1, 2, 3, 2, 2, 1).$$
$${\mathbf E}_7: (1, 2, 3, 4, 2, 3, 2). \,\, {\mathbf E}_8:  (2, 3, 4, 5, 6, 3, 4, 2).$$
Then for a non-exceptional locally smooth curve $D$ on $X$, $\mult_p(\pi(D)) = D.\xi$ (\cite[Prop. 5.4]{B} or \cite[Corollary 2.3]{J2} in the multiplicity-$1$ case, which is what we need here).

Next, we claim that, given a divisor $F$ meeting each $E_i$ nonnegatively, there exists an effective exceptionally supported divisor $G$ such that $(F+G).E_i = 0$ for all but at most one $i$ and, if such an $i$ exists, $(F+G).E_i=1$ and $E_i$ appears with coefficient $1$ in $\xi$ (\cite[Prop. 5.5]{B}; that is, in the terminology of \cite{J2}, $E_i$ is {\em admissible}). We will prove this by considering each singularity type ${\mathbf A}_n, {\mathbf D}_n, {\mathbf E}_n$ separately. Following the established numbering scheme, we will denote the intersection multiplicities of a divisor with the $E_i$ as an ordered $n$-tuple; suppose then that $F.E_i=s_i, 1\le i\le n$ (and continue with this notation as we modify $F$).

For an ${\mathbf A}_n$, proceed by induction on $\sum s_i$. If this number is $\le 1$, we are done, so assume $\sum s_i \ge 2$. If there are two indices $i<j$ such that $s_i>0,s_j>0$, add $E_i+E{i+1} + \dots +E_j$ to $F$. This new divisor has its multiplicities with $E_i$ and $E_j$ reduced by $1$ but those with $E_{i-1}$ and $E_{j+1}$ -- if they exist -- increased by $1$. Rewrite $i$ for $i-1$ and $j$ for $j+1$ and ontinue the process until $i=1$ or $j=n$, at which point adding the divisor reduces $\sum s_i$ and we are done. If there is a single $i$ such that $s_i>2$, add $E_i$ and reduce to the previous case.

For a ${\mathbf D}_n$, first suppose $n=4$. Add copies of $E_1, E_3, E_4$ as needed until $s_1, s_3, s_4$ are $\le 1$ and then add $\xi_0$ until $s_2=0$. If only one of $s_1,s_3,s_4$ is $1$, we are done. If (without loss of generality) $s_1=s_3=1,s_4=0$, add $(1,1,1,0)$. If $s_1=s_3=s_4=1$, add $(2,3,2,2)$. For $n>4$, $E_2, \dots, E_n$ form a ${\mathbf D}_{n-1}$, so by induction we can assume that at most one of $s_2, s_{n-1}, s_n$ is $1$ and all other $s_i, i>1$ are $0$ (note that these transformations can only increase $s_1$). Adding $(2,2, \dots, 2,1,1)$ decreases $s_1$ by $2$ and leaves the others unchanged, so assume $s_1$ is $1$ or $0$. Add $\xi_0$ if necessary so that $s_2=0$. If without loss of generality $s_1=s_{n-1}=1, s_n=0$, add $(1,1, \dots, 1,0)$. 

For an ${\mathbf E}_6$, note that we can add $\xi_0$ as needed to assume that $s_4=0$ without changing the other intersection numbers. Use the case ${\mathbf A}_5$ on the remaining $E_i$ to reduce at most one of the $s_i, i\neq 4$ to $1$. If $s_1 = 1$ or $s_6=1$, we are done. If without loss of generality $s_2=1$, add $(1,2,2,1,1,0)$, and if $s_3=1$, add $(1,2,3,0,2,1)$.

For an ${\mathbf E}_7$, note that $\xi_0.E_i = -\delta_i$, so we can use $\xi_0$ to reduce $s_7$ as needed. Use the $\rdD_6$ case on $E_1,\dots, E_6$ to reduce to three cases. If $s_1=1$, we are done. If $s_5=1$, add $(0,1,2,3,2,2,1)$. If $s_6=1$, add $(2,4,6,8,4,6,3)$.

For an ${\mathbf E}_8$, note that we can reduce $s_1$ to $0$ while leaving the others unchanged by adding $\xi_0$. Thus by treating the ${\mathbf E}_7$ formed by $E_2, \dots, E_8$, we see that the only case to consider is $s_2=1$ and $s_i=0$ for $i\neq 2$. In this case add $(3, 6, 8, 10, 12, 6, 8, 4)$.

Now let $\L$ be a very ample invertible sheaf on $X$ and let $D=\tilde{C}+G$ be a divisor satisfying (ii) above. If $D.E_i=0$ for all $i$, take $\M = \L$; otherwise, for the $i$-value for which $D.E_i=1$, let $\M = 2\L + (\L.E_i)E_i$. In this latter case, note that $\M$ is generated by global sections away from $E_i$ and, moreover, $\M$ is very ample off of $E_i$ and separates points off of $E_i$ from those on $E_i$. We claim that $\M$ is generated everywhere by global sections. Note that $\M|_{E_i} = \O_{E_i}$; one then can show by induction that $H^1(X,\M)=0$. Let $D'$ be a global section of $\M(D)$; then in the exact diagram

$$\begin{array}{ccccccccc}
& & 0 & & 0 & & 0 & & \\
& & \downarrow & & \downarrow & & \downarrow & & \\
0 & \to & \M(-E)  & \to & \O_X(D'-E) & \to & \O_D(D'-E) & \to & 0 \\
& & \downarrow & & \downarrow & & \downarrow & & \\
0 & \to & \M & \to & \O_X(D') & \to & \O_D(D') & \to & 0 \\
& & \downarrow & & \downarrow & & \downarrow & & \\
0 & \to & O_E  & \to & \O_E(D')  & \to & \O_{E\cap D'} (D') & \to & 0 \\
& & \downarrow & & \downarrow & & \downarrow & & \\
& & 0 & & 0 & & 0 & & 
\end{array}$$
the outer vertical columns are exact on global sections, so $H^1(X,\M(-E))=0$ since $H^1(X,\M)=0$; thus the map $H^1(X,\O_X(D'-E))\ra H^1(X, \O_X(D'))$ is injective, and therefore the middle column is exact on global sections as well. So $\O_X(D')$ is base-point free in all cases, and thus the general $D'$ is smooth in a neighborhood of $E$. By (i) above, then, its image $C'$ in $S$ is smooth at $p$. Moreover, as $D'=\tilde{C'}$ differs from $\tilde{C}$ only by a sum of exceptional divisors, all of which give class $0$ in $\Cl\Spec\O_{S,p}$, and $D'-D$, which is Cartier since it is a section of $\M$, the classes of $C'$ and $C$ are the same in $\Cl\Spec\O_{S,p}$.

Note that this does not say that we can produce a smooth curve whose strict transform meets any given $E_i$; {\em a priori} only those $E_i$ occurring with coefficient $1$ in $\xi$ are possible. However, there is an admissible representative for any element of the class group of the local ring.

Note also that if $S$ has only rational double point singularities, by repeating this process we can find a (globally) smooth curve having the same images in the class groups of all of the singular points as a given curve $C$.
\end{proof}

\end{document}